\documentclass[a4paper,11pt]{article}
\usepackage[latin1]{inputenc}
\usepackage[T1]{fontenc}
\usepackage{a4wide}
\usepackage{fontenc}
\usepackage{color}
\usepackage{amsmath}
\usepackage{amsthm} 
\usepackage{graphicx}
\usepackage{array}
\usepackage{amsfonts}
\usepackage{amssymb}
\usepackage{euscript}
\usepackage{delarray}
\usepackage{latexsym}
\usepackage{enumerate}
\usepackage{float}
\usepackage{algorithmicx}
\usepackage{algorithm}
\usepackage{algpseudocode}
\usepackage{url}

\title{Asymptotic normality of a Sobol index estimator in Gaussian process regression framework}
\author{Loic Le Gratiet $^\dag$  $^\ddag$\\ \\ $^\dag$ Universit\'e Paris Diderot 75205 Paris Cedex 13 \\ \\ $^\ddag$ CEA, DAM, DIF, F-91297 Arpajon, France  }

\newtheorem{theorem}{Theorem}

\newtheorem{prop}{Proposition}

\newcommand{\sigr}{\sigma_\varepsilon}

\newcommand{\EX}[1]{\mathbb{E}_X\left[#1\right]}
\newcommand{\VX}[1]{\mathrm{var}_X\left(#1\right)}
\newcommand{\covX}[2]{\mathrm{cov}_X\left(#1,#2\right)}
\newcommand{\EZ}[1]{\mathbb{E}_Z\left[#1\right]}

\newcommand{\N}[2]{\mathcal{N}\left(#1,#2\right)}

\newcommand{\probZ}[1]{\mathbb{P}_Z\left( #1 \right)}

\newcommand{\eqL}[0]{\stackrel{\mathcal{L}}{=}}
\newcommand{\tx}[0]{{\tilde{x}}}
\newcommand{\tX}[0]{{\tilde{X}}}

\newcommand{\zn}[0]{\mathbf{z}^n}

\newcommand{\K}[0]{\mathbf{K}}

\newcommand{\I}[0]{\mathbf{I}}

\newcommand{\D}[0]{\mathbf{D}}

\newcommand{\kk}[0]{\mathbf{k}}

\newcommand{\uu}[0]{\mathbf{u}}

\newcommand{\R}[0]{\mathbb{R}}

\newcommand{\mmu}[0]{\boldsymbol{\mu}}

\newcommand{\SSigma}[0]{\boldsymbol{\Sigma}}

\begin{document}
\maketitle

		\section{Abstract}

		Stochastic simulators  such as Monte-Carlo estimators are widely used in science and engineering to study physical systems through their probabilistic representation.    Global sensitivity analysis aims to identify the input parameters  which have the most important impact on the output. A popular tool  to perform global sensitivity analysis is the variance-based method which comes from the Hoeffding-Sobol decomposition. Nevertheless, this method requires an important number of simulations and is often unfeasible under reasonable time constraint. Therefore, an approximation of the input/output relation of the code is built with a Gaussian process regression model. This paper provides conditions which ensure the asymptotic normality of a Sobol's index estimator evaluated through this surrogate model. This result allows for building asymptotic confidence intervals for the considered Sobol index estimator. 
		The presented method is successfully applied on an academic example on the heat equation.

		\paragraph{Keywords:} Sensitivity analysis, Gaussian process regression, asymptotic normality, stochastic simulators, Sobol index.

		\section{Introduction}

		Complex computer codes usually  have a large number of input parameters.  
		 The determination of the important input parameters can be carried out by a global sensitivity analysis.
		We focus on the variance-based Sobol indices \cite{sobol1993}, \cite{Sal00},  \cite{sobol2} and \cite{cacuci} coming from the Hoeffding-Sobol decomposition \cite{hoeffding1948} which is valid when the input parameters are independent random variables.   For an extension of the Hoeffding-Sobol decomposition in a non-independent case, the reader is referred to  \cite{kucherenko}, \cite{daveiga}, \cite{mara}, \cite{li} and  \cite{chastaing}.

		Monte-Carlo methods are commonly  used   to estimate the Sobol indices (see \cite{sobol1993}, \cite{sobol2007estimating} and \cite{Jan12}). One of their  main advantages is that they  allow for quantifying the uncertainty related  to the estimation errors.  In particular, for non-asymptotic cases, this can be easily  carried  out with a bootstrap procedure  as presented in \cite{archer1997sensitivity} and  \cite{janon2011uncertainties}. Furthermore, in asymptotic cases, useful properties  can be shown as the asymptotic normality \cite{Jan12}. The reader is referred to   \cite{vaart98} for an extensive presentation of asymptotic statistics.
		Nevertheless, Monte-Carlo methods require a large number of simulations and are often  unachievable under  reasonable time constraints. Therefore,  in order to avoid prohibitive computational costs,  we surrogate the simulator with a  meta-model and we perform the estimations on it. 

		 In this paper, we consider a   special  surrogate model corresponding to  a Gaussian process regression.
		More precisely  we consider an idealized regression problem   for which we can deduce a posterior predictive mean and variance tractable for our purpose. In particular, we can  derive the  rate of convergence of the meta-model approximation error  with respect to the computational  budget. 

		Therefore, the Sobol index estimations - which are performed with a Monte-Carlo procedure by replacing the true code with the  posterior predictive mean - have two sources of uncertainty: the one due to the Monte-Carlo scheme and the one due to the meta-model approximation. The error due to the Monte-Carlo procedure tends to zero when the number of particles (calls of the meta-model)  tends to   infinity and the error due to the meta-model tends to zero when the computational  budget (calls of the complex simulator used to build the meta-model) tends to   infinity. A question of interest is whether the asymptotic normality presented in \cite{janon2011uncertainties}  is maintained.
		The principal difficulty of the study is that the estimator lies  in a product probability space which takes into account both  the uncertainty of the Gaussian process and the one of the Monte-Carlo sample.

		We emphasize that \cite{janon2011uncertainties} presents such a result for noise-free Gaussian process regression using a squared exponential covariance kernel (see \cite{R06}). They give conditions on the number of simulations and the number of Monte-Carlo particles which ensure the asymptotic normality for  the Sobol index estimators. A part of our developments is inspired by their work nevertheless they are different  with some important respects. Indeed, the particular case of noise-free Gaussian process regression with squared exponential covariance kernel allows for not considering the probability space in which lies the Gaussian process. This  significantly simplifies the mathematical  developments. Unfortunately this simplification does not hold in our  general framework.

		In this paper, we are interested in stochastic simulators which use Monte-Carlo or Monte-Carlo Markov Chain methods to solve a system of differential equations through its probabilistic interpretation.   Such simulators provide noisy observations  with a noise level inversely proportional to  the number of Monte-Carlo particles used by the simulator. 
		Therefore, with a fixed  computational budget,  we have to make a trade-off between the number of  simulations and the output accuracy. Actually, we   consider the asymptotic case where the number of observations is large. 

		The main result of this paper is a theorem giving sufficient conditions to ensure the asymptotic normality of the Sobol index estimators based on  the Monte-Carlo procedure presented in  \cite{sobol1993} and  using the presented Gaussian process regression model. We note that the presented theorem holds for a large class of covariance kernels. 
		The asymptotic normality is of interest since it allows for giving asymptotic confidence intervals on the Sobol index estimators.
		This result is illustrated with an academic example dealing with  the heat equation  problem.

		\section{Gaussian process regression for stochastic simulators}\label{boomboomboom}
		
		We present in Subsection \ref{asymptotic_presentation}  the practical problem that we want to deal with.  In order to handle the asymptotic  framework of a large number of observations, we replace the true problem by an idealized version of it in Subsection \ref{asymptotic_idealization}. This idealization allows us to  study the asymptotic normality of the Sobol's index estimator in Section \ref{AsymptoticNormality}.

			\subsection{Gaussian process regression  with a  large number of    observations}\label{asymptotic_presentation}
			
			Let us suppose that we want to surrogate a function $f(x) $, $x \in Q \subset \R^d$,  from noisy observations of it at points $(x_i)_{i=1,\dots,n}$  sampled from the probability measure $\mu$ - $\mu$ is called the design measure and $Q$ is an nonempty open set. 
			Furthermore, we consider that we have $r$ replications at each point.
			We hence  have $nr$ experiments of the form $z_{i,j}= f(x_i) + \varepsilon_{i,j}$, $i=1,\dots,n$, $j=1,\dots,r$ and we consider that $(\varepsilon_{i,j})_{\substack{i=1,\dots,n \\ j=1,\dots,r}}$ are  independently     sampled from a   Gaussian distribution with mean zero and variance $\sigr^2$.
			A stochastic simulator provides outputs  of the following  form 
			\begin{displaymath}
			z_i = \frac{1}{r}\sum_{j=1}^{r}z_{i,j} = f(x_i) + \varepsilon_i, \quad \forall i=1,\dots,n
			\end{displaymath}
			where $(\varepsilon_i)_{i=1,\dots,n}$ are  the observation noises  sampled from a zero-mean Gaussian distribution with variance  $\sigr^2/r$.
			Therefore, if we consider a fixed number of experiments $T=nr$,  we have an observation noise variance equal to $n\sigr^2/T$.
			
			Note that an observation noise variance   proportional to $n$ is natural in the framework of  stochastic simulators. Indeed, for a fixed total number of experiments  $T = nr$,  we can either decide to perform them in few points (i.e. $n$ small) but with lot of replications (i.e. $r$ large) or decide to perform them in lot of points (i.e. $n$ large) but with few replications (i.e. $r$ small).

			  In a Gaussian process regression framework,  we model $f(x)$ as   a Gaussian process  with a  known mean (that we take equal to zero without loss of generality) and a covariance kernel $k(x,\tx)$. Therefore, in the remainder of this paper, the function $f(x)$ is random. The predictive Mean Squared Error (MSE) of the Best Linear Unbiased Predictor (BLUP) given by
			\begin{equation}\label{BLUP_N}
			\hat{z}_{T,n}(x) = \kk'(x)  \left(\K+\frac{n\sigr^2}{T}\I\right)^{-1}\zn
			\end{equation}
			 is 
			\begin{equation}\label{MSE_BLUP_N}
			\sigma^2_{T,n}(x) = k(x,x) - \kk'(x) \left(\K+\frac{n\sigr^2}{T}\I\right)^{-1}\kk(x)
			\end{equation}
			where   $\zn = (z_i)_{i=1,\dots,n}$ denotes the vector of the observed values,   $k(x) = [k(x,x_i)]_{1\leq i \leq n}$ is the $n$-vector containing the covariances between $f(x)$ and $f(x_i), \quad 1\leq i \leq n$,  $\K = [k(x_i,x_j)]_{1\leq i,j \leq n}$ is the $n \times n$-matrix containing the covariances between  $f(x_i)$ and $f(x_j), \quad 1\leq i ,j \leq n$ and $\I$ is the $n \times n$ identity matrix. 
			
			In this paper, we consider the case  $n  \gg 1$. It corresponds to a massive experimental design set but with observations with a  large noise variance. This case is realistic for stochastic simulators where the computational cost resulting from  one Monte-Carlo particle is very low and thus  can be run in lot of points $(x_i)_{i=1,\dots,n}$.


			\subsection{Idealized Gaussian process regression} \label{asymptotic_idealization}

			We assume  from now on that  the positive kernel $k(x,\tx)$ is continuous and that $\sup_{x \in Q} k(x,x) < \infty$ where $Q$ is a nonempty open subset of $\R^d$. We introduce the Mercer's decomposition of $k(x,\tx)$   \cite{konig1986eigenvalue}, \cite{ferreira2009eigenvalues}:
			 \begin{equation}\label{Mercerk}
			k(x,\tx) = \sum_{p\geq0}\lambda_p \phi_p(x) \phi_p(\tx)
			\end{equation}
			where  $(\phi_p(x) )_p$ is an orthonormal  basis of $L^2_\mu(\mathbb{R}^d)$ consisting of eigenfunctions of the integral operator $(T_{\mu,k} g)(x)=\int_{\mathbb{R}^d} k(x,u)g(u)d\mu(u)$ and  $\lambda_p$ is the nonnegative sequence of corresponding eigenvalues sorted  in decreasing order. 

			Let us consider the following predictor:
			\begin{equation}\label{BLUP}
			 \hat{z}_T(x) = \sum_{p\geq0}\frac{\lambda_p}{\lambda_p + \sigr^2/T }z_p\phi_p(x)
			\end{equation}
			where $z_p = f_p + \varepsilon_p^*$, $f_p = \int f(x) \phi_p(x)\,d\mu(x)$,  $\varepsilon_p^* \sim \N{0}{\sigr^2/T }$, $\varepsilon_p^*$ independent of  $\varepsilon_q^*$ for $ p \neq q$ and     $(\varepsilon_p^*)_{p\geq 0}$ independent of $(f_p)_{p\geq 0}$. Note that we have $f_p \sim \N{0}{\lambda_p}$,  $f_p$ independent of $f_q$ for $p \neq q$  and $f(x) = \sum_{p \geq 0} f_p \phi_p(x)$.  

			Let us introduce  the probability space $(\Omega_Z, \mathcal{F}_Z, \mathbb{P}_Z) = (\Omega_f \times \Omega_\varepsilon, \sigma(\mathcal{F}_f \times \mathcal{F}_\varepsilon), \mathbb{P}_f \times \mathbb{P}_\varepsilon)$ where $(\Omega_f, \mathcal{F}_f, \mathbb{P}_f)$ corresponds to the probability space where  $f(x)$ and the sequence $(f_p)_{p \geq 0}$ are defined and $(\Omega_\varepsilon, \mathcal{F}_\varepsilon, \mathbb{P}_\varepsilon)$ is the probability space where the observation noises $(\varepsilon_{i})_{i \in \mathbb{N}}$ and  the sequence $(\varepsilon_p^*)_{p \geq 0}$ are  defined.
			Further, let us consider  the sequence of  independent random variables  $(X_i)_{i\in \mathbb{N}}$  with probability measure $\mu$ on $Q \subset \R^d$ and defined on  the probability space $(\Omega_D, \mathcal{F}_D, \mathbb{P}_D)$.   The sequence $(X_i)_{i = 1,\dots,n}$ represents the experimental design set considered as a random variable. 
			Therefore,   the predictors $\hat{z}_{T,n}(x)$ in (\ref{BLUP_N}) and $ \hat{z}_T(x)$ in (\ref{BLUP}) are associated to the random experimental design set $(X_i)_{i\in \mathbb{N}}$.
			We have the following convergence in probability  when $n \rightarrow \infty$ \cite{LoicRegularity}:  
			\begin{equation}\label{convmseD}
			\sigma^2_{T,n}(x) 
\underset{n \rightarrow \infty}{\overset{\mathbb{P}_D}{\longrightarrow}} \sigma^2_{T}(x) 
			\end{equation}
			where $\sigma^2_{T,n}(x) = \EZ{ (\hat{z}_{T,n}(x) - f(x))^2}$ (\ref{MSE_BLUP_N}) and  $\sigma^2_T(x) =  \EZ{(\hat{z}_T(x)-f(x))^2}$.
			Therefore $ \hat{z}_T(x)$ in (\ref{BLUP})  is a relevant  candidate for an idealized version of $\hat{z}_{T,n}(x)$ in (\ref{BLUP_N}) for  the considered  asymptotics $n \rightarrow \infty$. The following proposition  allows for completing the justification of  the relevance of $\hat{z}_{T,n}(x)$.

			\begin{prop}\label{idealizedz}
			Let us consider $f(x)$ a Gaussian process of zero mean and covariance kernel $k(x,\tx)$,    $\hat{z}_{T,n}(x)$ in (\ref{BLUP_N}) and $ \hat{z}_T(x)$ in (\ref{BLUP}) both associated to the random experimental design set  $(X_i)_{i \in \mathbb{N}}$. Consequently  $f(x) = \sum_{p \geq 0} f_p \phi_p(x)$ where $f_p \sim \N{0}{\lambda_p}$, $(f_p)_{p \geq 0}$ independent and $(\phi_p(x))_{p \geq 0}$ defined in (\ref{Mercerk}). The following convergence holds $\forall \delta > 0$ and for any  Borel set $A \subset  \R^2$ such that the Lebesgue measure of  its boundary is  zero:
			\begin{equation}\label{convidealproba}
			 \mathbb{P}_D\left(\left| \probZ{  (\hat{z}_{T,n}(x),f(x))   \in  A} -  \probZ{  ( \hat{z}_T(x), f(x))  \in A} \right|> \delta \right)  \stackrel{n \rightarrow \infty}{\longrightarrow} 0
			\end{equation}
			\end{prop}
	
			\begin{proof}[Proof of Proposition \ref{idealizedz}]
			First of all, we note that for a fixed $   \omega_D \in    \Omega_D$   the random variables $( \hat{z}_{T,n}(x), f(x)) $ and $( \hat{z}_{T}(x), f(x)) $ are  Gaussian since  they are  linear transformations of  $((\varepsilon_i)_{i  \in \mathbb{N}},(f_p)_{p\geq 0})$ and  $((\varepsilon_p^*)_{p \geq 0},(f_p)_{p\geq 0})$    which are both independently distributed from  Gaussian distributions.

			Thanks to the equality $\EZ{ (\hat{z}_{T,n}(x))^2} =  k(x,x) - \sigma^2_{T,n}(x)$ with $k(x,x) = \sum_{p \geq 0} \lambda_p \phi_p(x)^2$,  to the definition of  $\hat{z}_T(x)$ in (\ref{BLUP}) and to the convergence (\ref{convmseD}), the following convergence holds in probability when $n \rightarrow \infty$:
			\begin{equation}
			\EZ{ (\hat{z}_{T,n}(x))^2}  \underset{n \rightarrow \infty}{\overset{\mathbb{P}_D}{\longrightarrow}}
 \EZ{ (\hat{z}_{T}(x))^2} 
			\end{equation}
			Furthermore, we also have the equality  $\EZ{  \hat{z}_{T,n}(x)f(x)}  =  k(x,x) - \sigma^2_{T,n}(x)$ that leads   the convergence for $n \rightarrow \infty$:
			\begin{equation}
			\EZ{  \hat{z}_{T,n}(x)f(x)}  \stackrel{\mathbb{P}_D}{\longrightarrow}  \EZ{  \hat{z}_{T}(x)f(x)}
			\end{equation}		
			 We can deduce the following convergence of the covariance of the two-dimensional Gaussian vector $ ( \hat{z}_{T,n}(x),f(x))$ to the one of the two-dimensional  Gaussian vector $ ( \hat{z}_T(x),f(x))$  when $n \rightarrow \infty$:
			\begin{equation}\label{CntoC}
			\mathrm{cov}_Z\left(   ( \hat{z}_{T,n}(x),f(x)) \right) \stackrel{\mathbb{P}_D}{\longrightarrow} \mathrm{cov}_Z\left( ( \hat{z}_T(x),f(x))\right)
			\end{equation}			
			Furthermore,  the following equality  holds:
			\begin{equation}
			\EZ{   ( \hat{z}_{T,n}(x),f(x))   }  =  \EZ{ ( \hat{z}_T(x),f(x))} = (0,0)
			\end{equation}	
			Let us denote by 	$C_n = \mathrm{cov}_Z\left(   ( \hat{z}_{T,n}(x),f(x)) \right) $, for all Borel sets  $A  \subset  \R^2$ such that $\nu(\partial A )=0$ ($\nu$ denotes the Lebesgue measure and $\partial A $ the boundary of $A$), we have the following equality almost surely with respect to $(\Omega_D, \mathcal{F}_D, \mathbb{P}_D)$:
			\begin{equation*}
			 \probZ{ ( \hat{z}_{T,n}(x), f(x))  \in A} = \phi_2\left( C_n^{-1/2} A\right)
			\end{equation*}
			where $\phi_2$ stands for the bivariate  normal distribution $\N{0}{\I_2}$. We note that $C_n$ is a random variable defined on  the probability space $(\Omega_D, \mathcal{F}_D, \mathbb{P}_D)$. 
			Let us denote by $C =  \mathrm{cov}_Z\left( ( \hat{z}_T(x),f(x))\right)$. The matrix $C$ being nonsingular, the convergence (\ref{CntoC}) implies the following one when $n \rightarrow \infty$:
			\begin{equation*}
			C_n^{-1/2}  \underset{n \rightarrow \infty}{\overset{\mathbb{P}_D}{\longrightarrow}} C^{-1/2}
			\end{equation*}
			Therefore, for all Borel sets  $ A \subset \R^2$ such that $\nu(\partial A )=0$, we have when $n \rightarrow \infty$:
			\begin{equation*}
			\phi_2(C_n^{-1/2}  A)  \underset{n \rightarrow \infty}{\overset{\mathbb{P}_D}{\longrightarrow}}  \phi_2(C^{-1/2}  A)
			\end{equation*} 
			Finally, we can deduce that $\forall \delta > 0$ and for all Borel sets  $ A \subset \R^2$ such that $\nu(\partial(A)=0$, the   convergence in (\ref{convidealproba}) holds.
			\end{proof}

			The function $\hat{z}_T(x)$ is the surrogate model that we consider in this paper. We note that  $\hat{z}_T(x)$ is not equal to  the objective function $f(x)$ since $\sigr^2 /T \neq 0$.  In practical applications, we expect that the idealized model (\ref{BLUP})  is close enough to the actual surrogate model (\ref{BLUP_N}) so that it provides relevant  confidence  intervals. 

			Note that with this formalism $f(x)$ is a random process defined on the probability space  $(\Omega_Z, \mathcal{F}_Z, \mathbb{P}_Z)$. The random series $(z_p)_{p\geq 0}$ is defined on $(\Omega_Z, \mathcal{F}_Z, \mathbb{P}_Z)$ as well.
			In order to study the convergence of $\hat{z}_T(x) $ to the real function $f(x)$, let us consider  the following  equality:
			\begin{equation}\label{conv_MSE}
			\sigma^2_T(x)   =  \sum_{p \geq 0} \frac{ \sigr^2  \lambda_p/T}{\sigr^2 /T + \lambda_p} \phi_p(x)^2
			\end{equation}
			Then, let us define the Integrated Mean Squared Error (IMSE):
			\begin{equation}\label{IMSE}
			\mathrm{IMSE}_T = \int_{\mathbb{R}^d} \sigma_T^2(x) \, d\mu(x)= \EZ{||\hat{z}_T(x) - f(x) ||^2_{L^2_\mu}}
			\end{equation}
			The following equality holds:
			\begin{equation}\label{conv_IMSE}
			\mathrm{IMSE}_T  = \sum_{p \geq 0} \frac{ \sigr^2  \lambda_p/T}{\sigr^2 /T + \lambda_p}
			\end{equation}
			We can link the asymptotic rate of convergence of the IMSE (\ref{conv_IMSE}) with the asymptotic decay of the eigenvalues $(\lambda_p)_{p \geq 0}$ thanks to   the following inequalities \cite{LoicRegularity}:
			\begin{equation}\label{double_inequality}
			B_T^2 /2 \leq \mathrm{IMSE}_T \leq B_T^2
			\end{equation}
			with:
			\begin{equation}\label{BT}
			B_T^2 = \sum_{p \, \mathrm{s.t.} \, \lambda_p \leq \sigma^2_\varepsilon/T} \lambda_p + \frac{\sigma^2_\varepsilon}{T}\# \{ p \, \mathrm{s.t.} \, \lambda_p > \sigma^2_\varepsilon/T  \}
			\end{equation}
			
%

		\section{Asymptotic normality of a  Sobol  index estimator}\label{AsymptoticNormality}
		
		We present in this section the main theorem of this paper about the asymptotic normality of a  Sobol index estimator using Monte-Carlo integrations and the meta-model $\hat{z}_T(x)$ presented in Subsection \ref{asymptotic_idealization}. 
		In the forthcoming development, we suppose that $T$ is an increasing  sequence indexed by  the number of Monte-Carlo $m$  particles used to estimate the variance  and covariance terms involved in  the Sobol  index.  We use the notation $T_m$ to emphasize that $T$  depends on $m$.
		First of all, let us define  in Subsection \ref{indexdef} the Sobol indices and the considered Monte-Carlo  estimator.  

			\subsection{The  Sobol indices}\label{indexdef}
			
			Let us suppose that the input parameter is a random vector $X$ with probability  measure $\mu = \mu_1 \otimes  \mu_2$ on $(\R^{d_1} \times \R^{d_2}, \mathcal{B}(\R^{d_1} \times \R^{d_2}))$ with $d = d_1+d_2$.
			We consider the random vector $(X, \tX)$  defined   on the probability space $(\Omega_X, \mathcal{F}_X, \mathbb{P}_X)$  with $X = (X^1, X^2)$ and $\tX = (X^1, \tX^2)$ where $X^1$ is a random vector with values in $\R^{d_1}$ and  with distribution  $\mu_1$, $X^2$ and $\tX^2$ are random vectors with values  in $\R^{d_2}$ with distribution  $\mu_2$,  and $X^1$, $X^2$ and $\tX^2$ are independent.
			
			We are interested in the following closed Sobol index of parameter $X^1$  (see \cite{sobol1993}, \cite{Sal00}):
			\begin{equation}\label{SOBIND}
			S^{X^1} =  \frac{V^{X^1}}{V} = \frac{\VX{\EX{f(X)|X^1}}}{\VX{f(X)}} = \frac{\covX{f(X)}{ f(\tX)}}{\VX{f(X)}}
			\end{equation}
			where the random variables $f(X)$ and $f(\tX)$ are defined on  the   product  probability space $ (\Omega_Z \times \Omega_X, \sigma \left( \mathcal{F}_Z \times \mathcal{F}_X \right), \mathbb{P}_Z \times \mathbb{P}_X )$ and $S^{X^1}$, ${V^{X^1}}$ and ${V}$ are  defined  on  the probability space $(\Omega_Z, \mathcal{F}_Z, \mathcal{F}_Z)$. The Sobol index ${S^{X^1}}$ can be simply interpreted as a measure of the part of variance of $f(x)$ explained by the factor $X^1$. We note that $\VX{}$, $\EX{}$, $\covX{}{}$ stand for the variance, the expectation and the covariance in the probability space $(\Omega_X, \mathcal{F}_X, \mathbb{P}_X)$.
			 
			Furthermore, let us consider  the sequence $(X_i, \tX_i)_{i=1}^\infty$   of random variables defined  on  $(\Omega_X, \mathcal{F}_X, \mathbb{P}_X)$   independent and identically distributed such that $(X_i,\tX_i) \eqL (X,\tX)$ for all $i \in \mathbb{N}^*$ ($\eqL$ stands for the equality in distribution).   We use the following classical Monte-Carlo estimator for (\ref{SOBIND}) (see \cite{sobol1993}):
			\begin{equation}\label{SOBOLf}
					S^{X^1}_m = \frac{V_{ m}^{X^1}}{V_{ m}}=\frac{m^{-1}\sum_{i=1}^m{f(X_i)f(\tX_i)} - 
			m^{-2}\sum_{i,j=1}^mf(X_i)f(\tX_j)
			}{
			m^{-1}\sum_{i=1}^mf^2(X_i) - m^{-2}(\sum_{i=1}^mf(X_i))^2
			}
			\end{equation}
			where the random variables $S^{X^1}_m$, ${V_{ m}^{X^1}}$ and  ${V_{ m}}$ are defined on the   probability space $ (\Omega_Z \times \Omega_X, \sigma \left( \mathcal{F}_Z \times \mathcal{F}_X \right),$ $ \mathbb{P}_Z \times \mathbb{P}_X )$.

			Furthermore,  after substituting $f(x)$ with the meta-model  $\hat{z}_{T_m}(x)$, we obtain the following estimator:
			\begin{equation}\label{SOBOLz}
			S^{X^1}_{T_m,m} =  \frac{V_{T_m,m}^{X^1}}{V_{T_m,m}}=\frac{m^{-1}\sum_{i=1}^m{\hat{z}_{T_m}(X_i)\hat{z}_{T_m}(\tX_i)} - 
			m^{-2}\sum_{i,j=1}^m\hat{z}_{T_m}(X_i)\hat{z}_{T_m}(\tX_j)
			}{
			m^{-1}\sum_{i=1}^m\hat{z}_{T_m}^2(X_i) - m^{-2}(\sum_{i=1}^m\hat{z}_{T_m}(X_i))^2
			}
			\end{equation}
			where  the random variables $S^{X^1}_{T_m,m} $, ${V_{T_m,m}^{X^1}}$,  ${V_{T_m,m}}$, $\hat{z}_{T_m}(X_i)$ and $\hat{z}_{T_m}(\tX_j)$ are defined on the product probability space $ (\Omega_Z \times \Omega_X, \sigma \left( \mathcal{F}_Z \times \mathcal{F}_X \right), \mathbb{P}_Z \times \mathbb{P}_X )$.

			\subsection{Theorem on the asymptotic normality of the Sobol index estimator}
			
			The  theorem below gives the  relation between $T_m$ and $m$ which ensures the asymptotic normality of the estimator $ S^{X^1}_{T_m,m}$ when $m  \rightarrow \infty$.
			We note that $S^{X^1}_{T_m,m}$ is the  estimator of the Sobol  index $ S^{X^1} = \covX{f(X)}{ f(\tX)}/\VX{f(X)}$ when we replace the true function by the surrogate model  (\ref{BLUP}) and when we use the  Monte-Carlo estimator (\ref{SOBOLf}) for the variance and covariance  involved in the Sobol index.
			
			\begin{theorem}[Asymptotic normality of $S^{X^1}_{T_m,m}$]\label{normalityasymp}
			Let us consider the estimator  $ S^{X^1}_{T_m,m}$ (\ref{SOBOLz}) of $ S^{X^1} $  (\ref{SOBIND}) with $T_m$ an increasing function of $m \in \mathbb{N}^*$. We have the following convergences:
			\begin{itemize}
			\item[] If $mB^2_{T_m} \stackrel{m\rightarrow  \infty}{\longrightarrow}0$, then for all interval $I \in \R$ and $\forall \delta > 0$, we have the   convergence:
			\begin{equation}
			\mathbb{P}_Z \left( \left| \mathbb{P}_X \left(\sqrt{m} \left( S^{X^1}_{T_m,m} - S^{X^1} \right) \in I \right) -  \int_I{g(x)dx} \right| > \delta \right) \stackrel{m \rightarrow \infty}{\longrightarrow} 0
			\end{equation}
			where $g(x)$ is the probability density function of a zero-mean Gaussian random variable  with variance:
			\begin{equation}\label{variancesobolasymptotic}
			 \frac{\VX{\left( f(X) - \EX{f(X)}\right)\left( f(\tX) - \EX{f(X)} - S^{X^1} f(X) + S^{X^1} \EX{f(X)} \right)}}{\left( \VX{f(X)}\right)^2}
			\end{equation}
			with    $B^2_{T_m}$  given by  (\ref{BT}).
			\item[] If $mB^2_{T_m} \stackrel{m\rightarrow  \infty}{\longrightarrow}\infty$, then $\forall \delta > 0$, $\exists C > 0$ such that :
			\begin{equation}
			\mathbb{P}_Z \left( \left| \mathbb{P}_X \left( B^{-1}_{T_m}  \left( S^{X^1}_{T_m,m} - S^{X^1} \right)  \geq C  \right)- 1 \right| >  \delta \right)  \stackrel{m \rightarrow \infty}{\longrightarrow}  0
			\end{equation}
			\end{itemize}
			\end{theorem}
			
			Theorem \ref{normalityasymp} is of interest since it gives how fast $T_m$ has to increase with respect to $m$ so that the error of the surrogate modelling and the one of the Monte-Carlo sampling have the same order of magnitude. Indeed, for a given size $m$ of the Monte-Carlo sample, it is not necessary to use a too large $T_m$ otherwise the Monte-Carlo  estimation error will   dominate (it corresponds to  the case $mB^2_{T_m} \stackrel{m\rightarrow  \infty}{\longrightarrow}0$). On the other hand, if $T_m$ is taken too large  (it corresponds to the case $mB^2_{T_m} \stackrel{m\rightarrow  \infty}{\longrightarrow}\infty$), the estimation error is   dominated   by  the meta-model approximation.  
			
			Furthermore, we see that when $mB^2_{T_m} \stackrel{m\rightarrow  \infty}{\longrightarrow}0$, the asymptotic normality is assessed for the estimator $S^{X^1}_{T_m,m}$ with an explicit  variance given in equation (\ref{variancesobolasymptotic}). By studying in  (\ref{variancesobolasymptotic}) the cases $S^{X^1} = 0$ and $S^{X^1} = 1$ we  see  that the given estimator is more precise for large values of Sobol indices than for small ones. A  more efficient estimator   for small index values is given in \cite{sobol2007estimating}.
			
			We show in Section \ref{SOBOLex} that the product $mB^2_{T_m}$ can easily be handled when we have an  explicit  formula for the  asymptotic decay of the eigenvalues of the Mercer's decomposition of $k(x,\tx)$. The proof of Theorem \ref{normalityasymp} is given in   Appendix \ref{ploufplouf}. It is based on the Skorokhod's representation theorem \cite{Bil99}, the Lindeberg-Feller central limit theorem, and the Delta method \cite{vaart98}.

		\section{Examples of asymptotic normality  for Sobol's index}\label{SOBOLex}
		
		According to the previous developments, the desired asymptotic normality is assessed under the   assumption $m B_{T_m}^2 \stackrel{m \rightarrow \infty}{\longrightarrow} 0$. In the remainder of this section, we present relations between $T_m$ and $m$ which lead the convergence  $m B_{T_m}^2 \stackrel{m  \rightarrow \infty}{\longrightarrow} 0$  for    some usual kernels.
		
			\subsection{Asymptotic normality with  $d$-tensorised  Mat\'ern-$\nu$ kernels}
			
			We focus here on the d-tensorised Mat\'ern-$\nu$ kernel with regularity parameter $\nu > 1/2$ \cite{S99}, \cite{R06}:
			\begin{displaymath}
			k(x,\tx) = \prod_{i=1}^d  \frac{2^{1-\nu}}{\Gamma(\nu)} \left(\frac{\sqrt{2\nu}|x^i - \tx^i| }{\theta_i} \right)^\nu K_\nu \left( \frac{\sqrt{2\nu}|x^i - \tx^i| }{\theta_i} \right)
			\end{displaymath}
			where $K_\nu$ is the modified Bessel function \cite{Ab65}.
			 The eigenvalues   of this kernel satisfy the following asymptotic behavior \cite{PU11}:
			\begin{displaymath}
			\lambda_p =\phi(p) , \quad p \gg 1
			\end{displaymath}
			where $\phi(p) =  \left( \mathrm{log}(1+p)^{2(d-1){(\nu+1/2)}} \right)p^{-2{(\nu+1/2)}}\left(1 + \mathrm{O}(1/p)\right) $. 
			Therefore, for $T_m \gg 1$:
			\begin{displaymath}
			   B_{T_m}^2  \approx  \mathrm{log}(T_m/\sigr^2)^{d-1}\left(\frac{\sigr^2}{T_m}\right)^{1-1/2{(\nu+1/2)}}
			\end{displaymath}
			Section \ref{AsymptoticNormality} suggests that the asymptotic normality of the Sobol's index estimator is assessed when:
			\begin{displaymath}
			m B_{T_m}^2  \stackrel{m}{\longrightarrow} 0
			\end{displaymath}
			Let us consider that $T_m$ is such that:
			\begin{equation}\label{criticalpointmatern}
			 \mathrm{log}(T_m/\sigr^2)^{d-1}\left(\frac{\sigr^2}{T_m}\right)^{1-1/2{(\nu+1/2)}} = 1/m 
			\end{equation}
			It corresponds to the critical point $m B_{T_m}^2 \approx 1$.  In this  case, the error originates both from  the meta-model  approximation error and  the Monte-Carlo estimation error.
			Equation (\ref{criticalpointmatern}) leads to the following critical budget: 
			\begin{equation}\label{critmat}
			\frac{T_m}{\sigr^2} = \sigr^2  m^{1/(1-1/2{(\nu+1/2)})}\mathrm{log}\left( m \right)^{(d-1)}, 
			\end{equation}
			and, the asymptotic normality is assessed for:
			\begin{equation}\label{eq63}
			\frac{T_m}{\sigr^2} = \sigr^2  m^{1/(1-1/2{(\nu+1/2)})+\alpha}\mathrm{log}\left( m \right)^{(d-1)}, \, \forall \alpha > 0
			\end{equation}
			In practice, we want to minimize the budget allocated to the simulator and thus consider the case $\alpha$ tends to zero. As a consequence, for applications we will consider the allocation of the critical point (\ref{critmat}).
			
			\subsection{Asymptotic normality for   $d$-dimensional Gaussian kernels}\label{asympSOBOLgaussex}
			
			Let us consider the $d$-dimensional Gaussian kernel:
			\begin{equation}\label{GaussianKernel}
			k(x,\tx) = \exp \left( -\frac{1}{2} \sum_{i=1}^d \frac{(x^i - \tx^i)^2}{\theta_i^2}\right)
			\end{equation}			
			 Thanks to \cite{To06}, we have the following upper bound for   the eigenvalues:
			\begin{equation} \label{EigenGauss}
			\lambda_p \leq c' \mathrm{exp}\left( -c p^{1/d}\right) 
			\end{equation}
			with $c$ and $c'$ constants.
			From this inequality, we can deduce that $\exists C > 0$ such that:
			\begin{displaymath}
			 B_{T_m}^2  \approx C \mathrm{log}(T_m/\sigr^2)^{d}\left(\frac{\sigr^2}{T_m}\right) 
			\end{displaymath}
			Therefore, the critical budget corresponding to the critical point $m B_{T_m}^2 \approx 1$   is  given by 
			\begin{equation} \label{critGauss}
			T_m/\sigr^2 =  m  \mathrm{log}\left(m\right)^{d} 
			\end{equation}
			and the asymptotic normality for the Sobol index estimator  is assessed with:
			\begin{equation} 
			T_m/\sigr^2 =  m^{1+\alpha} \mathrm{log}\left(m\right)^{d} ,\, \forall \alpha > 0
			\end{equation}
			We note that the condition  is only sufficient  since we have an inequality in (\ref{EigenGauss}).

			\subsection{Asymptotic normality for  $d$-dimensional Gaussian kernels with a Gaussian measure $\mu(x)$}\label{asympSOBOLgaussex2}
	
			Let us consider a Gaussian measure $\mu \sim \mathcal{N}(0,\sigma_\mu^2\I)$ in dimension $d$ and the   Gaussian kernel (\ref{GaussianKernel}).
			As presented in \cite{zhu1998gaussian}, we have analytical expressions for the  eigenvalues and eigenfunctions of $k(x,\tx)$:
			\begin{displaymath}
			\lambda_p = \prod_{i=1}^d\sqrt{\frac{2a}{A_i}}B_i^p
			\end{displaymath}
			\begin{displaymath}
			\phi_p(x) =  \exp \left(-\sum_{i=1}^d (c_i-a)(x^i)^2 \right) \prod_{i=1}^d H_p( \sqrt{2c_i}x^i)
			\end{displaymath}
			where $H_p(x) = (-1)^p \exp (x^2) \frac{d^p}{dx^p}\exp (-x^2)$ is the $p^\mathrm{th}$ order Hermite polynomial (see \cite{gradshteuin2007table}),  $a = 1/(2\sigma_\mu)^2$, $b_i = 1/(2\theta_i^2)$ and
			\begin{displaymath}
			c_i = \sqrt{a^2 + 2ab_i}, \quad A_i = a+ b_i + c_i,  \quad B_i = b_i/A_i.
			\end{displaymath}
			Therefore,  the  eigenvalues      satisfy the following asymptotic behavior
			\begin{equation} \label{EigenGaussmu}
			\lambda_p  \propto \exp \left( -   p    \xi_d \right)
			\end{equation}
			where $\xi_d =  \sum_{i=1}^d \log \left(1/B_i\right) $.
			For $T_m \gg 1$, we have:
			\begin{equation}\label{BTGauss}
			B_{T_m}^2 \approx  \left( {\sigr^2}/{T_m}\right) \log \left(  {T_m}/{\sigr^2}  \right)/\xi_d
			\end{equation}
			Let us consider the     critical point $
			B_{T_m}^2 = 1/m $.
			Then,  the critical  budget is given by
			\begin{displaymath}
			 \frac{T_m}{\sigr^2} = \xi_d m \log (m)
			\end{displaymath}
			and  the asymptotic normality is assessed for:
			\begin{equation}\label{eq63}
			\frac{T_m}{\sigr^2} = \xi_d m^{1+\alpha} \log (m),  \, \forall \alpha > 0
			\end{equation}

		\section{Numerical illustration}
		
		The purpose of this section is to perform a global sensitivity analysis  of a stochastic code solving  the following heat equation:
		\begin{equation}
		\frac{\partial u}{\partial t}(x,t) - \frac{1}{2}  \Delta  u(x,t) = 0 
		\end{equation}
		with $x \in \mathbb{R}^d$ and  $u(x,0) = g(x) =  \exp(-\sum_{i=1}^dx_i^2/(2\sigma_{g,i}^2))$. 
		The function $u(x,t)$ has   the following probabilistic representation:
		\begin{equation}
		u(x,t) = \mathbb{E}_{W_t}[g(x+W_t)]
		\end{equation}
		where $W_t$ is the 1-dimensional  Brownian motion.
		We evaluate the function $u(x,t)$ through the following stochastic code:
		\begin{equation}
		u^\mathrm{code}_r(x,t) =\frac{1}{r} \sum_{i=1}^r  \left( \frac{1}{s}\sum_{j=1}^{s}g(x+W_{t,i,j}) \right)
		\end{equation}
		where the number of replications $r$ tunes the precision of the output, $s = 30$  and $(W_{t,i,j})_{\substack{i=1,\dots,r\\j=1,\dots,s}}$ are sampled from  a Gaussian random variable  of mean zero and variance $t$.
		
		We note that there is a closed  form expression for the  solution of the considered heat equation, that will allow is to compute exactly the Sobol indices and to assess the quality of our estimate:
		\begin{equation}
		u(x,t) = \prod_{i=1}^d   \left( \frac{\sigma_{g,i}^2}{\sigma_{g,i}^2+t} \right)^{1/2} \exp \left( -\frac{x_i^2}{2(\sigma_{g,i}^2 + t)} \right)
		\end{equation}
		
			\subsection{Exact Sobol indices}
			
			Let us consider that $x$ is a random variable $X $ defined on $ (\Omega_X, \mathcal{F}_X, \mathbb{P}_X)$ such that  $X \sim \mathcal{N}\left( 0, \sigma_\mu^2 \I\right)$. We are interested for the application in the first order Sobol indices, i.e. the contribution of $(X^j)_{j=1,\dots,d}$. 
			By straightforward calculations it can be shown that:
			\begin{equation}\label{exactSj}
			S^{X^j} =  \frac{V^{X^j}}{V}= \frac{\mathrm{var}_X(\mathbb{E}_X[u(X,t)|X^j])}{\mathrm{var}_X(u(X,t))} = \frac{B_j - 1 }{\left(\prod_{i=1}^d B_i \right)-1}
			\end{equation}
			where $X^j$ is the $j^{\mathrm{th}}$ component of the random vector $X$ with $j=1,\dots,d$ and 
			\begin{displaymath}
			B_j = \sigma_\mu \left(\frac{2}{t} - \frac{2}{t^2}\left( \frac{1}{t} + \frac{1}{\sigma_{g,i}^2} \right)^{-1} + \frac{1}{\sigma_\mu^2} \right)^{-\frac{1}{2}} \left( \frac{1}{t} + \frac{1}{\sigma_\mu^2} - \frac{1}{t^2}\left( \frac{1}{t} + \frac{1}{\sigma_{g,i}^2}\right) ^{-1} \right)
			\end{displaymath}
			Therefore, the importance measure of the $j^\mathrm{th}$ input is directly linked with the dispersion parameter $\sigma_{g,i}^2$ of the function $g(x)$. Furthermore, when $t$ tends to the infinity, the response $u(x,t)$ tends to zero as the variance of the main effect.  In this section, we consider the response   at $t=1$.
			
			\subsection{Model selection}
			
			Let us consider a Gaussian process of covariance $k_u(x,\tx)$ and  mean $m_u$ to surrogate $u(x,t)$ at $t=1$. We consider the predictive mean and variance presented in equations (\ref{BLUP_N}) and (\ref{MSE_BLUP_N}).
			As the response $u(x,t)$ is smooth,  we choose a squared exponential covariance kernel:
			\begin{displaymath}
			k_u(x,\tx) = \sigma^2 \exp \left( -\frac{1}{2} \sum_{i=1}^d \frac{(x^i - \tx^i)^2}{\theta_i^2} \right)
			\end{displaymath}
			Furthermore, as $u(x,t)$ tends to zero when $x$  tends to the infinity, we consider  that $m_u = 0$. Indeed, we want that the model tends   to   zero when we  move  away from the design points.
			
			The   experimental design  set $\D$   is composed of $n = 3000$ training points $x_i^\mathrm{train}$ sampled from the multivariate normal distribution $\mathcal{N}\left( 0, \sigma_\mu^2 \I\right)$  with $\sigma_\mu=2$ and $d=5$. Furthermore, the initial budget is $T_0 = 3000$.  It corresponds to a unique repetition $r_0=1$ at each point of $\D$.  The $n$ observations of $u^\mathrm{code}_{r_0}(x,1)$ at points in $\D$ are denoted by $\uu^n$.

			The hyper-parameters $\sigma^2$, $\theta$ and $\sigr^2$ are estimated by maximizing the marginal Likelihood \cite{R06}:
			\begin{displaymath}
			-\frac{1}{2}  \left( \uu^n\right)' \left( \sigma^2 \K + \sigr \I\right)^{-1} \uu^n - \frac{1}{2} \det \left( \sigma^2 \K + \sigr \I\right)
			\end{displaymath}
			where $\K = [k_u(x_i,x_j)]_{i,j=1,\dots,n}$. To solve the maximization problem, we have first randomly generated a set of 1,000 parameters $(\sigma^2, \theta, \sigr)$ on the domain $(0,10)\times(0,2)^d \times(0,1)$ and we have started a quasi-Newton based maximization from the 10 best parameters using the BFGS method.
			We obtain the following parameter estimations.
			\begin{itemize}
			\item  $\hat{\theta} =  \begin{pmatrix}1.01 &1.02& 1.03 &1.00& 1.07  \end{pmatrix}$
			\item  $\hat \sigma^2  =  1.46$
			\item $\hat \sigma_\varepsilon^2 = 6.74.10^{-2}$
			\end{itemize}
			Furthermore, the dispersion term of $g(x)$ are set to:
			\begin{itemize}
			\item  $(\sigma_{g,i}^2)_{i=1,\dots,d}  =    (5,3,2,1,1)$
			\end{itemize}
			
			\subsection{Convergence of IMSE$_T$}
			As presented in Subsection \ref{asymptotic_idealization} and Section \ref{AsymptoticNormality}, the asymptotic normality of the Sobol index estimator  is  closely related to the convergence of the generalization error IMSE$_T$ (\ref{IMSE}). Therefore, in order to effectively estimate the confidence intervals of the  estimators, we have to   characterize  this convergence.
			Especially, we have to take into account the initial budget used to select the model.
			The  value of  IMSE$_{T_0}$  where  $T_0$   corresponds to the initial budget  allocated to $\D$ is estimated to $\mathrm{IMSE}_{T_0} = 6.06.10^{-1}$.
			According to (\ref{BTGauss}), we have the following convergence rate for   IMSE$_T$ with respect to $T$:
			\begin{displaymath}
			\mathrm{IMSE}_T  \sim \left( {\sigr^2}/{T}\right) \log \left(  {T}/{\sigr^2}  \right)/\xi_d
			\end{displaymath}
			Therefore, from an initial budget $T_0$ we   expect that   IMSE$_T$   as a function of $T$ decays as:
			\begin{displaymath}
			\mathrm{IMSE}_T = \mathrm{IMSE}_{T_0} \frac{T_0 \log \left(  {T}/{\sigr^2}  \right)}{T  \log \left(  {T_0}/{\sigr^2}  \right)}
			\end{displaymath}
			The critical ratio $m B_T^2 = 1$ presented in Section \ref{SOBOLex} leads to the following budget:
			\begin{equation}\label{budgetapplic}
			T =\frac{m}{C} \log \left( \frac{m}{C \sigr^2 } \right)
			\end{equation}
			with $
			C =     \log \left(  {T_0}/{\sigr^2} \right)/(T_0 \mathrm{IMSE}_{T_0})
			$.

			\subsection{Confidence intervals for  the Sobol index estimations}
			According to Theorem \ref{normalityasymp}, if $T$ follows the relation in (\ref{budgetapplic}), the Sobol index estimator  presented in Subsection \ref{indexdef}  is asymptotically distributed with respect to a Gaussian random variable centered on the true index and with   variance given in (\ref{variancesobolasymptotic}).
			We use this property to build $90\%$ confidence intervals on the  estimations of $(S^j)_{j=1,\dots,d}$   (\ref{exactSj}). The exact values of the  Sobol indices (\ref{exactSj}) are given by:
		\begin{displaymath}
		(S^j)_{j=1,\dots,d} = (0.052, 0.088,  0.124, 0.194, 0.194)
		\end{displaymath}

			 Remember  that $m$ represents the number of  particles  for the Monte-Carlo integrations and $T$ is the budget used to construct the surrogate model $\hat z_T(x)$. 	
			In order to illustrate   the relevance of (\ref{budgetapplic}), we  consider the following equation:
			\begin{displaymath} 
			T = \sigr^2 \frac{m^\alpha}{C} \log \left( \frac{m}{C} \right)
			\end{displaymath}
			with different values of $\alpha$ - the right value being $\alpha=1$ - and  different  values of $m$. For each combination $(\alpha, m)$, we estimate the Sobol indices  with the estimator (\ref{SOBOLz})  and from 500 different Monte-Carlo samples $(x_i^{\mathrm{MC}})_{i=1,\dots,m}$. For each sample we evaluate the $90\%$ confidence intervals thanks to (\ref{variancesobolasymptotic}) and we check if  the estimations  are covered or not. The result of the procedure is presented in  Table \ref{tablesobol}.

		\begin{table}[H]
		\begin{center}
		\begin{tabular}{|c|c||c|c|c|c|c|}
		\hline
		$m$ & $\alpha$ & $S^1$ & $S^2$ & $S^3$ & $S^4$ & $S^5$ \\
		\hline
		\hline
		1,000 & 0.8 & 88.00 & 86.20 &  87.60  & 88.20 &  86.40  \\
		1,000 & 0.9 & 89.00 & 91.80 & 89.60 & 86.20 & 86.00 \\
		1,000 & 1.0 & 88.40 & 87.00 & 89.40 & 87.60 & 90.80 \\
		1,000 & 1.1 & 88.00 & 89.40 & 88.80 & 87.00 & 88.60  \\
		1,000 &  1.2 & 90.00 & 91.00  & 86.60 & 88.80  &89.00 \\
		\hline
		\hline
		3,000 & 0.8 & 88.00 & 87.60   &  86.60& 87.80 & 87.20 \\
		3,000 & 0.9 & 89.80  & 87.80  & 87.40 & 88.60  & 88.00 \\
		3,000 & 1.0 &  89.40  & 90.40 & 89.20 &  89.40   & 89.60  \\
		3,000 & 1.1 & 90.40 &90.60 & 91.00 & 91.60& 90.80   \\
		3,000 & 1.2 & 92.00  & 91.80& 92.00 &  91.40 &  91.40 \\
		\hline
		\hline
		5,000 & 0.8 & 87.60 & 86.20 & 87.40 &  88.20& 86.40 \\
		5,000 & 1.0 & 89.20 & 89.40  & 90.80 & 89.80 & 89.60  \\
		5,000 & 1.2 & 92.00& 91.40 & 92.80 & 90.60  & 92.20 \\
		\hline
		\end{tabular}
		\end{center}
		\caption{Coverage rates for $(S^j)_{j=1,\dots,d}$ in percentage. The confidence intervals are built from the variance presented in (\ref{variancesobolasymptotic})    in Theorem \ref{normalityasymp}. The theoretical rates is $90\%$ and the estimations is performed from 500 different Monte-Carlo samples.}
		\label{tablesobol}
		\end{table}

		We see in Table \ref{tablesobol} that the asymptotic behavior is not reached for $m=1,000$ Monte-Carlo particles since the coverage is  globally  too low in this case  for every $\alpha$. Furthermore, for $m=3,000$ and $m= 5,000$, we see that the coverage is globally better for $\alpha = 1$ than for  the other values. Indeed, the covering rate is  underestimated for $\alpha < 1$ and often  overestimated for $\alpha > 1$ whereas it is always around $90\%$ for $\alpha =1$.
		Furthermore, the confidence intervals seem to be well evaluated either for large values of $S^j$ with $S^4$ and $S^5$,  for  intermediate values of $S^j$ with $S^3$  or for  small values of   $S^j$ with $S^1$ and $S^2$. 
		 Therefore, this example emphasizes the relevance of the asymptotic normality for the Sobol index estimators presented in Theorem \ref{normalityasymp}.

		\section{Conclusion}

		This paper focuses  on the estimation of the Sobol indices to perform    global sensitivity analysis for  stochastic simulators.
		We suggest an index estimator which combines a Monte-Carlo scheme to estimate the integrals involved in the index definition and a Gaussian process regression to surrogate the   stochastic simulator.  The surrogate model is necessary since the Monte-Carlo integrations require  an important number of simulations. 

		In a  stochastic simulator  framework, for a fixed computational budget the observation noise variance is   inversely proportional to the number of simulations.  In this paper, we consider the   special case of  a large number of observations with an important uncertainty on the output.
		This choice   allows us to consider an idealized version of the regression problem from which  we can define a surrogate model which is tractable for our purpose.

		 In particular we aim to build confidence intervals for the index estimator  taking into account both the uncertainty due to the Monte-Carlo integrations and the one due to the surrogate modelling. 
		To handle this point, we present a theorem providing sufficient conditions to ensure the asymptotic normality of the suggested estimator. The proof of the theorem is the main point of this paper. It gives a closed form expression  for the variance of the  asymptotic distribution of the estimator.  From it we can easily estimate  the desired confidence  intervals.
		Furthermore, a strength of the suggested theorem is that it gives the relation between the number of   particles for the Monte-Carlo integrations and the computational budget allocated to the surrogate model so that they have the same contribution on the error of the Sobol index estimations.

		\section{Aknowledgments}

		The author is grateful  to his supervisor  Dr. Josselin Garnier  for his fruitful guidance and constructive suggestions.
		
\appendix

			\section{Proof of Theorem \ref{normalityasymp}}\label{ploufplouf}

			 Let us  denote by $S_{T_m}^{X^1} = \covX{\hat{z}_{T_m}(X)}{ \hat{z}_{T_m}(\tX)} /\VX{\hat{z}_{T_m}(X)}$  the   variance of the main effect of $X^1$ for  the surrogate model $\hat{z}_{T_m}(x)$ (\ref{BLUP}).
			The random variables    $S^{X^1}$ and  $S_{T_m}^{X^1}$  are defined on  the probability space  $(\Omega_Z  , \mathcal{F}_Z \ , \mathbb{P}_Z  ) $ 
			 and the random variables $S^{X^1}_{T_m,m}$, $\hat{z}_{T_m}(X)$ and  $f(X)$ are defined on the product probability space $(\Omega_Z \times \Omega_X, \sigma(\mathcal{F}_Z \times \mathcal{F}_X), \mathbb{P}_Z \otimes  \mathbb{P}_X) $.

			Let us consider the following decomposition:
			\begin{equation}
			S^{X^1}_{T_m,m}-S^{X^1}  =  S^{X^1}_{T_m,m} - S_{T_m}^{X^1} + S_{T_m}^{X^1} -S^{X^1} 
			\end{equation}
			In a first hand we deal with the convergence of $\sqrt{m}\left( S^{X^1}_{T_m,m} -  S_{T_m}^{X^1} \right)$. We handle this problem thanks to the Skorokhod's representation theorem, the  Lindeberg-Feller theorem and the Delta method. In a second hand, we study the convergence of $\sqrt{m} \Big( S_{T_m}^{X^1} - S^{X^1}  \Big)$ through the Skorokhod's representation theorem.
			
			In the forthcoming developments, we consider that $mB_{T_m}^2 \stackrel{m \rightarrow \infty}{\longrightarrow} 0$. Therefore,  there exists $ g(T_m)$ such that $  g(T_m) \stackrel{m \rightarrow \infty}{\longrightarrow} 0 $ and  $mB_{T_m}^2g^{-2}(T_m) \stackrel{m \rightarrow \infty}{\longrightarrow} 0$.  The function $g(T_m)$ considered in the remainder of this section satisfies    this property.

				\subsection{The Skorokhod's representation theorem}  Let us consider the following random  variables defined on  the probability space $(\Omega_Z, \mathcal{F}_Z, \mathbb{P}_Z)$:
				\begin{equation}\label{aTm}
				a_{T_m}(x) = (\hat{z}_{{T_m}}(x) - f(x))B_{T_m}^{-1}g(T_m)
				\end{equation}
				\begin{equation}\label{bTm}
				b_{T_m}(x) = (\hat{z}_{{T_m}}(x) - f(x))g(T_m) ^{1/3}B_{T_m}^{-1/3}
				\end{equation}
				Markov's inequality and (\ref{double_inequality}) give  us $\forall \delta > 0$:
				\begin{displaymath}
				\mathbb{P}_Z(||a_{T_m}(x)||^2_{L^2_\mu} > \delta) \leq \mathbb{E}_Z(||a_{T_m}(x)||^2_{L^2_\mu} ) /\delta \leq g(T_m)^2/\delta
				\end{displaymath}
				Therefore, we have the following convergence in probability  in $(\Omega_Z, \mathcal{F}_Z, \mathbb{P}_Z)$:
				\begin{displaymath}
				\lim_{m \rightarrow \infty}   ||a_{T_m}(x)||^2_{L^2_\mu}  = 0
				\end{displaymath}
				and the inequalities in (\ref{double_inequality})  ensure the following one:
				\begin{equation}\label{lowera}
				||a_{T_m}(x)||^2_{L^2_\mu} \geq g(T_m)^2/2
				\end{equation}
				Furthermore, the following equality stands since $f(x)$ is a Gaussian process:
				\begin{displaymath}
				\mathbb{E}_Z[(\hat{z}_{T_m}(x) - f(x))^6] = 15 \sigma_{T_m}^6(x)
				\end{displaymath}
				   Cauchy-Schwarz inequality   leads to:
				\begin{displaymath}
				\mathbb{E}_Z[||\hat{z}_{T_m}(x) - f(x)||^6_{L^6_\mu}] \leq 15   \int \sigma_{T_m}^6(x)\,d\mu(x) \leq  15   B_{T_m}^2 \sup_x k^2(x,x)
				\end{displaymath}
				Therefore, thanks to    Markov's inequality we have:
				\begin{displaymath}
				\mathbb{P}_Z(||b_{T_m}(x)||^6_{L^6_\mu} > \delta) \leq 15  g(T_m)^2 \sup_x k^2(x,x)/\delta
				\end{displaymath}
				and the following convergence stands in probability   in $(\Omega_Z, \mathcal{F}_Z, \mathbb{P}_Z)$:
				\begin{displaymath}
				\lim_{m \rightarrow \infty}   ||b_{T_m}(x)||^6_{L^6_\mu}  = 0
				\end{displaymath}
				Therefore,  we have the following convergences in probability in $(\Omega_Z, \mathcal{F}_Z, \mathbb{P}_Z)$ when ${m} \rightarrow \infty$:
				\begin{displaymath}
				\left\{
				\begin{array}{l}
				f(x) \\
				a_{T_m}(x) = (\hat{z}_{T_m}(x) - f(x))g(T_m) B_{T_m}^{-1} \\
				b_{T_m}(x) = (\hat{z}_{T_m}(x) - f(x))g(T_m)^{1/3} B_{T_m}^{-1/3} \\
				\end{array}
				\right.
				\underset{{m} \rightarrow \infty}{\overset{L^6_\mu \times L^2_\mu \times L^6_\mu}{\longrightarrow}}
				\begin{pmatrix}
				f(x) \\
				0 \\
				0 \\
				\end{pmatrix}
				\end{displaymath}

				As $L^6_\mu \times L^2_\mu \times L^6_\mu$ is separable   we can use the Skorokhod's representation theorem \cite{Bil99} presented below.
				\begin{theorem}[Skorokhod's representation theorem]\label{skotheo}
				Let $\mu_n$, $n\in \mathbb{N}$ be a sequence of probability measures on a topological space $S$; suppose that $\mu_n$ converges weakly to some probability measure $\mu$ on $S$ as $n \rightarrow \infty$. Suppose also that the support of $\mu$ is separable. Then there exist random variables $X_n$ and $X$ defined on a common probability space $(\Omega, \mathcal{F}, \mathbb{P})$ such that:
				\begin{itemize}
				\item[(i)]  $\mu_n$ is the distribution of $X_n$
				\item[(ii)] $\mu$ is the distribution of $X$
				\item[(iii)] $X_n(\omega) \rightarrow X(\omega)$ as $n\rightarrow \infty$ for every $\omega \in \Omega$.
				\end{itemize}
				\end{theorem}

				Therefore,  there is a probability space denoted by  $(\tilde{\Omega}_Z, \tilde{\mathcal{F}}_Z, \tilde{\mathbb{P}}_Z)$ such that  					\begin{equation} (\tilde{f}_{T_m}(x), \tilde{a}_{T_m}(x), \tilde{b}_{T_m}(x)) \stackrel{\mathcal{L}}{=}  ({f}(x), {a}_{T_m}(x), {b}_{T_m}(x)), \quad \forall {m}
				\end{equation}
 				 with  $(\tilde{f}_{T_m}(x), \tilde{a}_{T_m}(x), \tilde{b}_{T_m}(x)) $, $\tilde{f}(x)$ defined on $ (\tilde{\Omega}_Z, \tilde{\mathcal{F}}_Z, \tilde{\mathbb{P}}_Z) $ and  $ ({f}(x), {a}_{T_m}(x), {b}_{T_m}(x))  $ defined on $  ( {\Omega}_Z,  {\mathcal{F}}_Z,  {\mathbb{P}}_Z)  $  -  and $ \forall \tilde{\omega}_Z \in  \tilde{\Omega}_Z$  the following convergence holds for ${m} \rightarrow \infty$:
				\begin{equation}
				(\tilde{f}_{T_m}(x), \tilde{a}_{T_m}(x), \tilde{b}_{T_m}(x))	
\underset{m \rightarrow \infty}{\overset{L^6_\mu \times L^2_\mu \times L^6_\mu}{\longrightarrow}}
				(\tilde{f}(x), 0, 0)
				\end{equation}
				First, let us build below  the analogous  of  ${z}_{T_m}(x)$ in $(\tilde{\Omega}_Z, \tilde{\mathcal{F}}_Z, \tilde{\mathbb{P}}_Z)$.  For a fixed $T_m > 0$, we have the equality 
				$a_{T_m}(x)g(T_m)^{-1}B_{T_m} = b_{T_m}(x)g(T_m)^{-1/3}B_{T_m}^{1/3}$. 
				Therefore, we have
				\begin{displaymath}
				||a_{T_m}(x)g(T_m)^{-1}B_{T_m} -  b_{T_m}(x)g(T_m)^{-1/3}B_{T_m}^{1/3}  ||_{L^2_\mu} = 0 
				\end{displaymath}
				and
				\begin{displaymath}
				\mathbb{P}_Z \left(||a_{T_m}(x)g(T_m)^{-1}B_{T_m} -  b_{T_m}(x)g(T_m)^{-1/3}B_{T_m}^{1/3}  ||_{L^2_\mu} = 0\right) = 1 
				\end{displaymath}
				The equality $(\tilde{a}_{T_m}(x), \tilde{b}_{T_m}(x)) \stackrel{\mathcal{L}}{=} (a_{T_m}(x), b_{T_m}(x))$ $\forall {T_m}$ leads to the following one
				\begin{displaymath}
				\tilde{\mathbb{P}}_Z \left(||\tilde{a}_{T_m}(x)g(T_m)^{-1}B_{T_m} - \tilde{b}_{T_m}(x)g(T_m)^{-1/3}B_{T_m}^{1/3}  ||_{L^2_\mu} = 0\right) = 1 
				\end{displaymath}
				Thus,  for almost every $\tilde{\omega_Z} $  in  $ \tilde{\Omega}_Z$, we have
				 \begin{equation}\label{aTm}
				||\tilde{a}_{T_m}(x)g(T_m)^{-1}B_{T_m} - \tilde{b}_{T_m}(x)g(T_m)^{-1/3}B_{T_m}^{1/3}  ||_{L^2_\mu}= 0
				\end{equation}
				 If we consider such a $\tilde{\omega}_Z$ we have  the equality $\tilde{a}_{T_m}(x)g(T_m)^{-1}B_{T_m} = \tilde{b}_{T_m}(x)g(T_m)^{-1/3}B_{T_m}^{1/3}$  for  $\mu$-almost every $x$.
				
				Let us denote by 
				$$\tilde{z}_{T_m}(x) =  \tilde{f}_{T_m}(x) + g(T_m)^{-1}B_{T_m}\tilde{a}_{T_m}(x),  $$ 
				$\tilde{z}_{T_m}(x)  $  is defined on $ (\tilde{\Omega}_Z, \tilde{\mathcal{F}}_Z, \tilde{\mathbb{P}}_Z) $. For $\tilde{\omega}_Z$ such that (\ref{aTm}) holds we  have the   equality $\tilde{z}_{T_m}(x)  = \tilde{f}_{T_m}(x) + g(T_m)^{-1/3}B_{T_m}^{1/3}\tilde{b}_{T_m}(x) $  
				for  $\mu$-almost every $x$.
				
				\subsection{Convergences with  a fixed $\tilde{\omega}_Z \in \tilde{\Omega}_Z$}
				
				Let us  consider a fixed  $\tilde{\omega}_Z \in \tilde{\Omega}_Z$ such that (\ref{aTm}) holds.  We aim to    study the convergence of ${\sqrt{m}\left(  \tilde{S}_{T_m,m}^{X^1}  -   \tilde{S}_{T_m}^{X^1} \right)}$ and  ${\sqrt{m}\left(  \tilde{S}_{T_m}^{X^1}-  \tilde{S}^{X^1}   \right)}$ in $( {\Omega}_X,  {\mathcal{F}}_X,  {\mathbb{P}}_X) $ with:
				\begin{equation}
				\tilde{S}^{X^1} = \mathrm{cov}_X(\tilde{f}(X), \tilde{f}(\tX))/\mathrm{var}_X(\tilde{f}(X)), 
				\end{equation}
				\begin{equation}
				\tilde{S}_{T_m}^{X^1} = \mathrm{cov}_X(\tilde{z}_{T_m}(X), \tilde{z}_{T_m}(\tX))/\mathrm{var}_X(\tilde{z}_{T_m}(X))
				\end{equation}
				and
				\begin{equation}
				\tilde{S}_{T_m,m}^{X^1} =\frac{m^{-1}\sum_{i=1}^n{\tilde{z}_{T_m}(X_i)\tilde{z}_{T_m}(\tX_i)} - 
				m^{-2}\sum_{i,j=1}^n\tilde{z}_{T_m}(X_i)\tilde{z}_{T_m}(\tX_j)
				}{
				m^{-1}\sum_{i=1}^n\tilde{z}_{T_m}^2(X_i) - m^{-2}(\sum_{i=1}^n\tilde{z}_{T_m}(X_i))^2
				}
				\end{equation}

				\subsubsection{Convergence of ${\sqrt{m}\left( \tilde{S}_{T_m,m}^{X^1}  -  \tilde{S}_{T_m}^{X^1} \right)}$  in $ {( {\Omega}_X,  {\mathcal{F}}_X,  {\mathbb{P}}_X)} $}

				Let us denote by $Y_{T_m,i} = \tilde{z}_{T_m}(X_i)$, $Y_{T_m,i}^{X^1} = \tilde{z}_{T_m}(\tX_i)$ and 
				\begin{equation}
				\begin{array}{ll}
				U_{T_m,i}  =  &  \left( (Y_{T_m,i}-\mathbb{E}_X[Y_{T_m,i}])(Y_{T_m,i}^{X^1}-\mathbb{E}_X[Y_{T_m,i}]),\right. \\
				 &\left. Y_{T_m,i}-\mathbb{E}_X[Y_{T_m,i}], Y_{T_m,i}^{X^1}-\mathbb{E}_X[Y_{T_m,i}], (Y_{T_m,i}-\mathbb{E}_X[Y_{T_m,i}])^2 \right)
				\end{array}
				\end{equation}
				Since $\tilde{\omega}_Z \in \tilde{\Omega}_Z$ is fixed,  $Y_{T_m,i}$, $Y_{T_m,i}^{X^1}$ and $U_{T_m,i} $ are defined on the probability space $( {\Omega}_X,  {\mathcal{F}}_X,  {\mathbb{P}}_X) $.
				For each $m$,  $(U_{T_m,i}/ \sqrt{m})_{i=1,\dots,m}$ is a sequence of independent random vectors such that for any  $ \varepsilon >0$:
				\begin{eqnarray*}
				\sum_{i=1}^m\mathbb{E}_X\left[||U_{T_m,i}||^2/m \mathbf{1}_{\{ ||U_{T_m,i}|| > \varepsilon \sqrt{m} \}}\right] & =   &
				\mathbb{E}_X\left[||U_{T_m,1}||^2  \mathbf{1}_{\{ ||U_{T_m,1}|| > \varepsilon \sqrt{m} \}}\right]
				 \\
				& \leq & 
				\mathbb{E}_X\left[||U_{T_m,1}||^{3}  \right]/( \varepsilon \sqrt{m}) 
				\end{eqnarray*} 
				since  $||U_{T_m,1}|| > \varepsilon \sqrt{m}$. \\
				
				We aim below to find an upper bound for  $\sup_{T_m}  \mathbb{E}_X\left[||U_{T_m,i}||^{3}  \right] $. 
				First, for any $m$,  let us consider the component $(Y_{T_m,i}-\mathbb{E}_X[Y_{T_m}])(Y_{T_m,i}^{X^1}-\mathbb{E}_X[Y_{T_m}]) $. We have the following inequality:
				\begin{displaymath}
				\mathbb{E}_X\left[ | (Y_{T_m,i}-\mathbb{E}[Y_{T_m,i}])(Y_{T_m,i}^{X^1}-\mathbb{E}[Y_{T_m,i}])|^{3} \right]  \leq C \mathbb{E}_X\left[ |  Y_{T_m,i} |^{6 } \right]
				\end{displaymath}
				with $C >0$ a constant.
				  Minkowski inequality and the equality   $\tilde{z}_{T_m}(x) =  (\tilde{f}_{T_m}(x) + g(T_m)^{-1/3}B_{T_m}^{1/3} \tilde{b}_{T_m}(x))$ for $\mu$-almost every $x$  give that  there exists $ C, C' > 0$ such that:
				\begin{eqnarray*}
				\mathbb{E}_X\left[ |  Y_{T_m,i} |^{6 } \right]& \leq &  
				 C || \tilde{f}_{T_m}(x) ||_{L^6_\mu}^{6 } + C' B_{T_m}^{2}g(T_m)^{-2}||   \tilde{b}_{T_m}(x) ||_{L^6_\mu}^{6 }
				\end{eqnarray*}
				The convergence  $(\tilde{f}_{T_m}(x), \tilde{b}_{T_m}(x)) \underset{{m} \rightarrow \infty}{\overset{L^6_\mu \times  L^6_\mu}{\longrightarrow}} (\tilde{f}(x), 0)$ implies  that   there exists $ C > 0$ such that  for any $m$:
				\begin{equation}\label{ineq1}
				\mathbb{E}_X\left[ | (Y_{T_m,i}-\mathbb{E}_X[Y_{T_m,i}])(Y_{T_m,i}^{X^1}-\mathbb{E}_X[Y_{T_m,i}])|^{3} \right]   \leq    C \\
				\end{equation}
				Second, following the same guideline, we find that  there exists $ C, C', C'' > 0$ such that for any $m$:
				\begin{equation}\label{ineq2}
				\mathbb{E}_X\left[ | (Y_{T_m,i}-\mathbb{E}_X[Y_{T_m,i}])^2|^{3} \right]  \leq    C  
				\end{equation}
				\begin{equation}\label{ineq3}
				\mathbb{E}_X\left[ | Y_{T_m,i}-\mathbb{E}_X[Y_{T_m,i}]|^{3} \right]  \leq    C'  
				\end{equation}
				\begin{equation}\label{ineq4}
				\mathbb{E}_X\left[ | Y_{T_m,i}^{X^1}-\mathbb{E}_X[Y_{T_m,i}]|^{3} \right]  \leq    C' 
				\end{equation}
				Third, the inequalities (\ref{ineq1}), (\ref{ineq3}),  (\ref{ineq3}) and  (\ref{ineq4})   give that $\sup_{T_m}  \mathbb{E}_X\left[||U_{T_m}||^{3}  \right]  < \infty$. 

				The inequality 
				$\sum_{i=1}^m\mathbb{E}_X\left[||U_{T_m,i}||^2/m \mathbf{1}_{\{ ||U_{T_m,i}|| > \varepsilon \sqrt{m} \}}\right] \leq 
				\mathbb{E}_X\left[||U_{T_m,1}||^{3}  \right]/( \varepsilon \sqrt{m})$ and the uniform  boundedness of $ \mathbb{E}_X\left[||U_{T_m}||^{3}  \right] $ lead to  the following convergence $\forall \varepsilon > 0$ when $m \rightarrow \infty$:
				\begin{equation}\label{LindevergFellerCond}
				\sum_{i=1}^m\mathbb{E}_X\left[||U_{T_m,i}||^2/m \mathbf{1}_{\{ ||U_{T_m,i}|| > \varepsilon \sqrt{m} \}}\right]  
				=
				\mathbb{E}_X\left[||U_{T_m,i}||^2  \mathbf{1}_{\{ ||U_{T_m,i}|| > \varepsilon \sqrt{m} \}}\right]  
				 \stackrel{m \rightarrow \infty }{\longrightarrow}  0 
				\end{equation} 
				and thus  $||U_{T_m,i}||^2$ is uniformly integrable. \\
				
				Now, we aim to show the convergence in probability of $U_{T_m,i} \stackrel{m \rightarrow \infty}{\longrightarrow} U_{i}$ in  $(\Omega_X, {\mathcal{F}}_X, {\mathbb{P}}_X)$. Let us denote by  
				\begin{displaymath}
				U_{i} = \left( (Y_{i}-\mathbb{E}_X[Y_{i}])(Y_{i}^{X^1}-\mathbb{E}_X[Y_{i}]),Y_{i}-\mathbb{E}_X[Y_{i}], Y_{i}^{X^1}-\mathbb{E}_X[Y_{i}], (Y_{i}-\mathbb{E}_X[Y_{i}])^2 \right)
				\end{displaymath}
				with $Y_i  = \tilde{f}(X_i )$ and $Y_i^{X^1}  = \tilde{f}(\tX_i )$. The random variables $U_{i} $, $Y_i $ and $Y_i^{X^1} $ are defined on $(\Omega_X, \mathcal{F}_X, \mathbb{P}_X)$ since $\tilde{\omega}_Z \in \tilde{\Omega}_Z$ is fixed.

				First, we study the term $\mathbb{E}_X\left[\left|U^{(1)}_{T_m,i} - U^{(1)}_i\right|\right]$ where 
				$
				 U^{(1)}_i = (Y_i - \mathbb{E}_X[Y_i])(Y_i^{X^1} - \mathbb{E}_X[Y_i])
				$ 
				and
				$
				 U^{(1)}_{T_m,i} =(Y_{T_m,i}-\mathbb{E}_X[Y_{T_m,i}])(Y_{T_m,i}^{X^1}-\mathbb{E}_X[Y_{T_m,i}])
				$.
				We have the following equality:
				\begin{eqnarray*}
				\mathbb{E}_X\left[\left|U^{(1)}_{T_m,i} - U^{(1)}_i\right|\right] & = & \mathbb{E}_X\left[\left|
				 \left(Y_{T_m,i}-\mathbb{E}_X[Y_{T_m,i}]\right)\left((Y_{T_m,i}^{X^1}-\mathbb{E}_X[Y_{T_m,i}])
				- (Y_{ i}^{X^1}-\mathbb{E}_X[Y_{ i}])
				 \right) \right.\right.  \\
				 & + & \left.\left.
				 (Y_{ i}^{X^1}-\mathbb{E}_X[Y_{ i}])\Big((Y_{T_m,i} -\mathbb{E}_X[Y_{T_m,i}])
				- (Y_{ i} -\mathbb{E}_X[Y_{ i}])
				 \Big) \right|\right] 
				\end{eqnarray*}
				from which we deduce the inequality:
				\begin{eqnarray*}
				\mathbb{E}_X\left[\left|U^{(1)}_{T_m,i} - U^{(1)}_i\right|\right] & \leq & \mathbb{E}_X\left[\left|
				 \left(Y_{T_m,i}-\mathbb{E}_X[Y_{T_m,i}]\right)\left((Y_{T_m,i}^{X^1}-\mathbb{E}_X[Y_{T_m,i}])
				- (Y_{ i}^{X^1}-\mathbb{E}_X[Y_{ i}])
				 \right) \right|\right]  \\
				 & + & \mathbb{E}_X\left[\left|
				 (Y_{ i}^{X^1}-\mathbb{E}_X[Y_{ i}])\Big((Y_{T_m,i} -\mathbb{E}_X[Y_{T_m,i}])
				- (Y_{ i} -\mathbb{E}_X[Y_{ i}])
				 \Big) \right|\right] 
				\end{eqnarray*}
				and from   Cauchy-Schwarz inequality   there exists $ C, C', C'' > 0$ such that:
				\begin{eqnarray*}
				\mathbb{E}_X\left[\left|U^{(1)}_{T_m,i} - U^{(1)}_i\right|\right] & \leq & C\mathbb{E}_X\left[ 
				 \left(Y_{T_m,i}-\mathbb{E}_X[Y_{T_m,i}]\right)^2\right]^{1/2} \mathbb{E}_X\left[(Y_{T_m,i}^{X^1} 
				- Y_{i}^{X^1})^2
				 \right] ^{1/2} \\
				 & + & C'  \mathbb{E}_X\left[ 
				 (Y_{i}^{X^1}-\mathbb{E}_X[Y_{i}])^2\right]^{1/2} \mathbb{E}_X\left[ (Y_{T_m,i} 
				- Y_{i} )^2  \right] ^{1/2} \\
				& \leq &  C''
				 \mathbb{E}_X\left[ (Y_{T_m,i} 
				- Y_{i} )^2  \right] ^{1/2}
				\left(
				\mathbb{E}_X\left[ 
				 (Y_{i}^{X^1})^2\right]^{1/2}  + \mathbb{E}_X\left[ 
				 \left(Y_{T_m,i}\right)^2\right]^{1/2}
				\right)
				\end{eqnarray*}
				The equality  $Y_{T_m,i} - Y_{i} = g(T_m)^{-1}B_{T_m}\tilde{a}_{T_m}(X_i )$ for $\mathbb{P}_X$-almost every $\omega_X \in \Omega_X$ implies  that   $\mathbb{E}_X\left[ (Y_{T_m,i} - Y_{i} )^2  \right] ^{1/2} = g(T_m)^{-1}B_{T_m}  \mathbb{E}_X\left[ (\tilde{a}_{T_m}(X_i))^2  \right]^{1/2}$.
				Since $\tilde{a}_{T_m}(x) \stackrel{m \rightarrow \infty }{\longrightarrow} 0$ in $L^2_\mu$, we have the convergence  $\mathbb{E}_X\left[ (Y_{T_m,i} - Y_{i} )^2  \right] ^{1/2} \stackrel{m \rightarrow \infty }{\longrightarrow} 0$. 
				
				Furthermore,  there exists $ C, C' > 0$ such that $\mathbb{E}_X\left[ 
				 (Y_{i}^{X^1})^2\right]^{1/2} < C$ and $ \mathbb{E}_X\left[ 
				 \left(Y_{T_m,i}\right)^2\right]^{1/2} < C'$ since $\tilde{z}_{T_m}(x) = \tilde{f}_{T_m}(x)+ g(T_m)^{-1}B_{T_m}\tilde{a}_{T_m}(x)$, $\tilde{f}_{T_m}(x) \stackrel{m \rightarrow \infty }{\longrightarrow} \tilde{f}(x)$ in $L^6_\mu$ and $\tilde{a}_{T_m}(x)  \stackrel{m \rightarrow \infty }{\longrightarrow}  0$ in $L^2_\mu$.
				Therefore, we have the following convergence:
				\begin{equation}\label{convU1}
				\mathbb{E}_X\left[\left|U^{(1)}_{T_m,i} - U^{(1)}_i\right|\right] \stackrel{m \rightarrow \infty}{\longrightarrow} 0
				\end{equation}
				Then, if we consider the terms
				$
				 U^{(4)}_i = (Y_i - \mathbb{E}_X[Y_i])^2
				$
				and
				$
				 U^{(4)}_{T_m,i} =(Y_{T_m,i}-\mathbb{E}_X[Y_{T_m,i}])^2
				$.
				Following the same guideline we find the convergence:
				\begin{equation}\label{convU4}
				\mathbb{E}_X\left[\left|U^{(4)}_{T_m,i} - U^{(4)}_i\right|\right] \stackrel{m \rightarrow \infty}{\longrightarrow} 0
				\end{equation}
				Furthermore, denoting by 
				$
				 U^{(2)}_i = (Y_i - \mathbb{E}_X[Y_i])
				$, 
				$
				 U^{(2)}_{T_m,i} =(Y_{T_m,i}-\mathbb{E}_X[Y_{T_m,i}])
				$, 
				$
				 U^{(3)}_i = (Y_i^{X^1} - \mathbb{E}_X[Y_i])
				$ and 
				$
				 U^{(3)}_{T_m,i} =(Y_{T_m,i}^{X^1}-\mathbb{E}_X[Y_{T_m,i}])
				$, 
				we have  the following inequalities:
				\begin{displaymath}
				\mathbb{E}_X\left[\left|U^{(2)}_{T_m,i} - U^{(2)}_i\right|\right]   \leq C\mathbb{E}_X\left[(Y_{T_m,i} - Y_{i})^2\right] ^{1/2}
				\end{displaymath}
				\begin{displaymath}
				\mathbb{E}_X\left[\left|U^{(3)}_{T_m,i} - U^{(3)}_i\right|\right]   \leq C'\mathbb{E}_X\left[(Y_{T_m,i}^{X^1} - Y_{i}^{X^1})^2\right] ^{1/2}
				\end{displaymath}
				with $C, C'$ positive constants. The convergences $\tilde{f}_{T_m}(x) \stackrel{L^6_\mu}{\rightarrow} \tilde{f}(x)$ and $\tilde{a}_{T_m}(x) \stackrel{L^6_\mu}{\rightarrow} 0$ when $ m \rightarrow \infty$ ensure that: 
				 \begin{equation}\label{convU2}
				\mathbb{E}_X\left[\left|U^{(2)}_{T_m,i} - U^{(2)}_i\right|\right] \stackrel{m \rightarrow \infty}{\longrightarrow} 0
				\end{equation}
				and
				\begin{equation}\label{convU3}
				\mathbb{E}_X\left[\left|U^{(3)}_{T_m,i} - U^{(3)}_i\right|\right] \stackrel{m \rightarrow \infty}{\longrightarrow} 0
				\end{equation}
				Finally,  the convergences presented in (\ref{convU1}), (\ref{convU4}), (\ref{convU2}) and (\ref{convU3}) imply  the desired one:
				\begin{equation}\label{convU}
				\mathbb{E}_X\left[ ||U_{T_m,i} - U_i ||\right] \stackrel{m \rightarrow \infty}{\longrightarrow} 0
				\end{equation}
				   Markov's inequality gives $\forall \delta >0$: 
				\begin{equation}\label{inequalityU}
				\mathbb{P}_X\left(||U_{T_m,i} - U_i||\geq \delta \right) \leq \mathbb{E}_X\left[||U_{T_m,i} - U_i||\right] /\delta
				\end{equation}
				The equations (\ref{convU})  and  (\ref{inequalityU}) imply the convergence $U_{T_m,i} \stackrel{m \rightarrow \infty}{\longrightarrow}  U_{i}$ in  probability in $(\Omega_X, \mathcal{F}_X, \mathbb{P}_X)$.
				
				 This  convergence in probability and the  uniform  integrability of $ ||U_{T_m,i} ||^2$  implies that $U_{T_m,i} \stackrel{m \rightarrow \infty}{\longrightarrow}  U_{i}$  in  ${L}^2(\Omega_X)$ and thus $\mathrm{cov}_X(U_{T_m,i}) \stackrel{m \rightarrow \infty}{\longrightarrow}  \mathrm{cov}_X(U_{i}) = \SSigma$.
				We note that we have also the convergence $\mathbb{E}_X[U_{T_m,i}]\rightarrow \mathbb{E}_X[U_{i}] =  \mmu$  since the convergence in $L^2(\Omega_X)$ implies the one in $L^1(\Omega_X)$.
				
				The  condition (\ref{LindevergFellerCond}) and the convergence $\sum_{i=1}^m\mathrm{cov}_X(U_{T_m,i})/m = \mathrm{cov}_X(U_{T_m,i}) \stackrel{m \rightarrow \infty}{\longrightarrow} \SSigma$ allow for using the Lindeberg-Feller  Theorem (see \cite{vaart98})  which ensures the following convergence  in $(\Omega_X, \mathcal{F}_X, \mathbb{P}_X)$:
				\begin{eqnarray*}
				\sum_{i=1}^m (U_{T_m,i}/ \sqrt{m} - \mathbb{E}_X[U_{T_m,i}/ \sqrt{m}]) & = &\sqrt{m} \left(\sum_{i=1}^m (U_{T_m,i})/m - \mathbb{E}_X[U_{T_m,i}] \right)  \\
				&\underset{m \rightarrow \infty}{\overset{\mathcal{L}}{\longrightarrow}}& \mathcal{N}\left(0,\SSigma\right)
				\end{eqnarray*}
				Furthermore, we have the following equality:
				\begin{displaymath}
				\tilde{S}_{T_m,m}^{X^1}  = \Phi(\bar{U}_{T_m})
				\end{displaymath}
				where $\bar{U}_{T_m} = \sum_{i=1}^m  U_{T_m,i}/m$ and $\Phi(x,y,z,t) = (x-yz)/(t-y^2)$. Therefore, the Delta method gives that  in $(\Omega_X, \mathcal{F}_X, \mathbb{P}_X)$:
				\begin{equation}\label{conv1}
				\sqrt{m} \left( \tilde{S}_{T_m,m}^{X^1} - \tilde{S}_{T_m}^{X^1}  \right)   \underset{m \rightarrow \infty}{\overset{\mathcal{L}}{\longrightarrow}} \mathcal{N}\left(0, \nabla \Phi^T(\mmu) \SSigma   \nabla \Phi(\mmu)\right)
				\end{equation}
				where $\mmu =  \mathbb{E}_X[U_i] =  \left( \mathrm{cov}_X(Y_i,Y_i^{X^1}), 0, 0, \mathrm{var}_X(Y_i)  \right)$. We note that the assumption $ \mathrm{var}_X(Y_i)  \neq 0$ justifies the use of the Delta method. A simple calculation gives that:
				\begin{equation}\label{AsymptoVariance}
				\nabla \Phi^T(\mmu) \SSigma   \nabla \Phi(\mmu) = \frac{\mathrm{var}_X\left(
				(Y_i-\mathbb{E}_X[Y_i])
				\left(
				Y_i^{X^1}-\mathbb{E}_X[Y_i ] - \tilde{S}^{X^1}Y_i + \tilde{S}^{X^1}\mathbb{E}_X[Y_i]
				\right)
				 \right)}{(\mathrm{var}_X(Y_i ))^2}
				\end{equation}
				with $\tilde{S}^{X^1} = \mathrm{cov}_X(Y_i,Y_i ^{X^1})/\mathrm{var}_X(Y_i )  = \mathrm{var}_X(\mathbb{E}_X[Y_i |X^1])/\mathrm{var}_X(Y_ i) $.
				
				\subsubsection{Convergence of ${\sqrt{m}\left(  \tilde{S}_{T_m}^{X^1}-  \tilde{S}^{X^1} \right)}$   in ${( {\Omega}_X,  {\mathcal{F}}_X,  {\mathbb{P}}_X)} $}
				
				Analogously to  \cite{Jan12}, we have the equality:
				\begin{eqnarray*}
				\tilde{S}_{T_m}^{X^1}-\tilde{S}^{X^1} & = &  \frac{\mathrm{var}_X(\tilde{\delta}_{T_m,i} )^{1/2}C_{\tilde{\delta}_{T_m,i}}}{\mathrm{var}_X(Y_i)+2\mathrm{cov}_X(Y_i,\tilde{\delta}_{T_m,i})+\mathrm{var}_X(\tilde{\delta}_{T_m,i})} 
				\end{eqnarray*}
				where $\tilde{\delta}_{T_m}(x) =  g(T_m)^{-1}B_{T_m} \tilde{a}_{T_m}(x)$,
				\begin{equation}\label{constCdT}
				\begin{array}{lll}
				C_{\tilde{\delta}_{T_m,i}} & =  & 2\mathrm{var}_X(Y_i)^{1/2}(\mathrm{cor}_X(Y_i,\tilde{\delta}_{T_m,i})-\mathrm{cor}_X(Y_i,Y_i^{X^1})\mathrm{cor}_X(Y_i,\tilde{\delta}_{T_m,i}))\\
				 & & +\mathrm{var}_X(\tilde{\delta}_{T_m,i})^{1/2}(\mathrm{cor}_X(\tilde{\delta}_{T_m,i},\tilde{\delta}_{T_m,i}^{X^1})-\mathrm{cor}_X(Y_i,Y_i^{X^1}))
				\end{array}
				\end{equation}
				  $\tilde{\delta}_{T_m,i} = \tilde{\delta}_{T_m,i}(X_i )$ and $\tilde{\delta}_{T_m,i}^{X^1} =\tilde{\delta}_{T_m,i}(\tX_i)$.
				The random variables  $\tilde{\delta}_{T_m,i}$ and  $\tilde{\delta}_{T_m,i}^{X^1}$  are defined on the product space $(\tilde{\Omega}_Z \times \Omega_X, \sigma( \tilde{\mathcal{F}}_Z \times  \mathcal{F}_X), \tilde{\mathbb{P}}_Z \otimes  \mathbb{P}_X) $  and $ \tilde{S}^{X^1}   $,  $\tilde{\delta}_{T_m}(x)$ and $C_{\tilde{\delta}_{T_m,i}} $ are defined on $ (\tilde{\Omega}_Z, \tilde{\mathcal{F}}_Z, \tilde{\mathbb{P}}_Z)$. We still consider a fixed $\tilde{\omega}_Z \in \tilde{\Omega}_Z$ such that (\ref{aTm}) holds.
				The  assumption $ \mathrm{var}_X(Y_i)  \neq 0$ ensures that the denominator is not equal to  zero and the convergences $\tilde{f}_{T_m}(x) \stackrel{L^6_\mu}{\rightarrow} \tilde{f}(x)$ and $\tilde{a}_{T_m}(x) \stackrel{L^2_\mu}{\rightarrow} 0$ give that $\sup_{ m} C_{\tilde{\delta}_{T_m,i}} < \infty$.
				Furthermore, since $\tilde{a}_{T_m}(x) \stackrel{L^2_\mu}{\rightarrow} 0$ we have the following inequalities:
				\begin{displaymath}
				\mathrm{var}_X(\tilde{\delta}_{T_m,i} )  \leq  C \mathbb{E}_X[(B_{T_m}g(T_m)^{-1}\tilde{a}_{T_m}(X_i))^2]  \leq C'  g(T_m)^{-2}B_{T_m}^2
				\end{displaymath}
				with $C, C'$ positive constants.

				Thanks to   Slutsky's theorem, the   convergence  $m g(T_m)^{-2}B_{T_m}^2 \stackrel{m}{\longrightarrow} 0$  ensures the following asymptotic normality when $m \rightarrow \infty$ in $(\Omega_X, \mathcal{F}_X, \mathbb{P}_X)$:
				\begin{equation}\label{AsymNorm}
				\sqrt{m}  \left( \tilde{S}_{T_m,m}^{X^1}-\tilde{S}^{X^1} \right)  \underset{m \rightarrow \infty}{\overset{\mathcal{L}}{\longrightarrow}} \mathcal{N}\left(0, \nabla \Phi^T(\mmu) \SSigma   \nabla \Phi(\mmu)\right)
				\end{equation}
				
				\subsubsection{The case  $ { m B_{T_m}^2 \stackrel{m \rightarrow \infty}{\longrightarrow} \infty }$.}
				
				Let us suppose that $m B_{T_m}^2 \stackrel{m \rightarrow \infty}{\longrightarrow} \infty $. We consider   the convergences of 
				\begin{equation}\label{convBTinf1}
				{{ B_{T_m}^{-1}}\left( \tilde{S}_{T_m,m}^{X^1}  -  \tilde{S}_{T_m}^{X^1} \right)}
				\end{equation}
				and
				\begin{displaymath}
				{{B_{T_m}^{-1}}\left(  \tilde{S}_{T_m}^{X^1}-  \tilde{S}^{X^1}    \right)}
				\end{displaymath}
				  in $(\Omega_X, \mathcal{F}_X, \mathbb{P}_X)$ with a fixed $\tilde{\omega}_Z \in \tilde{\Omega}_Z$ such that (\ref{aTm}) holds.
				We have the following equality:
				\begin{displaymath} 
				{{ B_{T_m}^{-1}}\left( \tilde{S}_{T_m,m}^{X^1}  -  \tilde{S}_{T_m}^{X^1} \right)}
				=
				{{ (\sqrt{m}B_{T_m})^{-1}}\sqrt{m} \left( \tilde{S}_{T_m,m}^{X^1}  -  \tilde{S}_{T_m}^{X^1} \right)}
				\end{displaymath}
				The convergence $(\sqrt{m}B_{T_m})^{-1}  \stackrel{m \rightarrow \infty}{\longrightarrow} 0$ and the convergence in (\ref{conv1}) (which does not  depend  on the convergence of the ratio  between $B_{T_m}^{-2}$ and $\sqrt{m}$)  imply the following one:
				\begin{displaymath} 
				{{ B_{T_m}^{-1}}\left( \tilde{S}_{T_m,m}^{X^1}  -  \tilde{S}_{T_m}^{X^1} \right)}
				\stackrel{m \rightarrow \infty}{\longrightarrow} 0
				\end{displaymath}
				Finally, thanks to the  inequality (\ref{lowera}),   there exists $ C, C' > 0$ such that
				\begin{eqnarray*}
				{{B_{T_m}^{-1}}\left(  \tilde{S}_{T_m}^{X^1}-  \tilde{S}^{X^1}    \right)}  
				  & = & {B_{T_m}^{-1}} \frac{g(T_m)^{-1}B_{T_m}\mathrm{var}_X(\tilde{a}_{T_m}(X_i) )^{1/2}C_{\tilde{\delta}_{T_m,i}}}{\mathrm{var}_X(Y_i)+2\mathrm{cov}_X(Y_i,\tilde{\delta}_{T_m,i})+\mathrm{var}_X(\tilde{\delta}_{T_m,i})} \\
				 & \geq  & C g(T_m)^{-1}  \frac{g(T_m)  C_{\tilde{\delta}_{T_m,i}} }{\mathrm{var}_X(Y_i)+2\mathrm{cov}_X(Y_i,\tilde{\delta}_{T_m,i})+\mathrm{var}_X(\tilde{\delta}_{T_m,i})}\\
				 & \geq  & C'C_{\tilde{\delta}_{T_m,i}}
				\end{eqnarray*}
				Therefore, if we have $C_{\tilde{\delta}_{T_m,i}} > 0$, the asymptotic normality is not reached and the estimator is biased.
				Regarding the expression of $C_{\tilde{\delta}_{T_m,i}}$ in (\ref{constCdT}) and assuming that $\mathrm{var}_X(Y_i) \neq  0$, $C_{\tilde{\delta}_{T_m,i}} = 0$ could happen if:
				\begin{itemize}
				\item $\mathrm{cor}_X(Y_i,Y_i^{X^1}) = 1$, i.e. all the variability of  $\tilde{f}(x)$  is explained by the variable $X^1$.
				\item  $\mathrm{var}_X(\tilde{\delta}_{T_m,i}) = 0$, i.e. the surrogate model error is null.
				\end{itemize}
				
				\subsection{Convergence in the  probability space ${( {\Omega}_Z \times \Omega_X ,  \sigma({\mathcal{F}}_Z \times \mathcal{F}_X),  {\mathbb{P}}_Z \otimes \mathbb{P}_X)}$.}
				
				We have proved   that for almost every $\tilde{\omega}_Z \in \tilde{\Omega}_Z$:
				\begin{itemize}
				\item[] If $m B_{T_m}^2 \stackrel{m \rightarrow \infty}{\longrightarrow} 0$, then
				\begin{displaymath}
				\forall I \in \R, \,\, \mathbb{P}_X\left( \sqrt{m}\left( \tilde{S}_{T_m,m}^{X^1}  -   \tilde{S}^{X^1}  \right) \in I \right) \stackrel{m  \rightarrow \infty}{\longrightarrow}  \int_I \tilde g(x)dx  
				\end{displaymath}
				\item[] If $m B_{T_m}^2 \stackrel{m \rightarrow \infty}{\longrightarrow} \infty$, then
				\begin{displaymath}
				\exists C > 0 \mathrm{\,s.t.}\,\,  \mathbb{P}_X\left( B_{T_m}^{-1}\left( \tilde{S}_{T_m,m}^{X^1}  -   \tilde{S}^{X^1}  \right) \geq C \right) \stackrel{m \rightarrow \infty}{\longrightarrow} 1  
				\end{displaymath}
				\end{itemize}
				where $\tilde g(x)$ is the probability density function of a random Gaussian vector of  zero mean and covariance $ \nabla \Phi^T(\mmu) \SSigma   \nabla \Phi(\mmu)$ (\ref{AsymptoVariance}).
				Therefore, in the probability space $(\tilde{\Omega}_Z \times \Omega_X, \sigma(\tilde{\mathcal{F}}_Z \times  \mathcal{F}_X), \tilde{\mathbb{P}}_Z \otimes  \mathbb{P}_X) $ we have 
				\begin{itemize}
				\item[] If $m B_{T_m}^2 \stackrel{m \rightarrow \infty}{\longrightarrow} 0$, then
				\begin{displaymath}
				\forall I \in  \R, \forall \delta >0, \,\, \tilde{\mathbb{P}}_Z \left( \left| \mathbb{P}_X\left( \sqrt{m}\left( \tilde{S}_{T_m,m}^{X^1}  -   \tilde{S}^{X^1}   \right) \in I \right) -  \int_I \tilde g(x)dx \right| > \delta \right) \stackrel{m \rightarrow \infty}{\longrightarrow} 0
				\end{displaymath}
				\item[] If $m B_{T_m}^2 \stackrel{m \rightarrow \infty}{\longrightarrow} \infty$, then
				\begin{displaymath}
				\forall \delta > 0, \exists C > 0 \mathrm{\,s.t.}\,\, \tilde{\mathbb{P}}_Z \left( \left|  \mathbb{P}_X\left( B_{T_m}^{-1}\left( \tilde{S}_{T_m,m}^{X^1}  -   \tilde{S}^{X^1}  \right) \geq C \right) - 1  \right| > \delta \right) \stackrel{m \rightarrow \infty}{\longrightarrow} 0
				\end{displaymath}
				\end{itemize}

				and the equalities  $(\tilde{f}_{T_m}(x), \tilde{a}_{T_m}(x), \tilde{b}_{T_m}(x)) \stackrel{\mathcal{L}}{=}  ({f}(x), {a}_{T_m}(x), {b}_{T_m}(x))$   and  $\tilde{f} (x) \stackrel{\mathcal{L}}{=} f(x)$ for all $m$ give  us   in the probability space $( {\Omega}_Z \times \Omega_X,  \sigma({\mathcal{F}}_Z \times  \mathcal{F}_X),  {\mathbb{P}}_Z \otimes  \mathbb{P}_X) $:
				\begin{itemize}
				\item[] If $m B_{T_m}^2 \stackrel{m \rightarrow \infty}{\longrightarrow} 0$, then
				\begin{displaymath}
				\forall I \in \Omega_X, \forall \delta > 0, \,\,  {\mathbb{P}}_Z \left( \left|  \mathbb{P}_X\left( \sqrt{m}\left(   {S}_{T_m,m}^{X^1}   -   {S}^{X^1}  \right) \in I \right) -  \int_I g(x)dx \right| > \delta \right)  \stackrel{m \rightarrow \infty}{\longrightarrow}  0
				\end{displaymath}
				\item[] If $m B_{T_m}^2 \stackrel{m \rightarrow \infty}{\longrightarrow} \infty$, then
				\begin{displaymath}
				\forall \delta >0, \exists C > 0 \mathrm{\,s.t.}\,\,  {\mathbb{P}}_Z \left( \left|  \mathbb{P}_X\left( B_{T_m}^{-1}\left(   {S}_{T_m,m}^{X^1}  -   {S}^{X^1}  \right) \geq C \right)  - 1   \right| > \delta \right) \stackrel{m \rightarrow \infty}{\longrightarrow} 0
				\end{displaymath}
				\end{itemize}		
				where $ g(x)$ is the probability density function of a random Gaussian vector of  zero mean and variance
				\begin{displaymath} 
				 \frac{\VX{\left( f(X) - \EX{f(X)}\right)\left( f(\tX) - \EX{f(X)} - S^{X^1} f(X) + S^{X^1} \EX{f(X)} \right)}}{\left( \VX{f(X)}\right)^2}
				\end{displaymath}
				This completes the proof.

\bibliographystyle{ieeetr}
\bibliography{biblio}

\end{document}